\documentclass[10pt]{article}
\usepackage{slashbox}
\usepackage{epsfig}

\newtheorem{theorem}{Theorem}

\newenvironment{conjecture}{{\bf Conjecture}}{\hspace*{1mm}}
\newtheorem{lemma}{Lemma}
\newtheorem{proposition}{Proposition}
\newtheorem{corollary}{Corollary}
\newtheorem{definition}{Definition}
\newtheorem{preproof}{{\bf Proof}}

\newenvironment{proof}[1]{\begin{preproof}{\rm #1}\hfill{\rule[-0.5mm]
                         {2mm}{2mm}}}{\end{preproof}}
\newtheorem{pretheorema}{{\bf Theorem}}

\newenvironment{theorema}[1]{\begin{pretheorema}
   {\hspace{-0.2em}{\rm #1}{\bf.}}}{\end{pretheorema}}

\newtheorem{prepropositiona}{{\bf Proposition}}

\input amssym.tex
\title{{\bf Independence number  of  generalized Petersen graphs}}

\author{ Nazli Besharati \thanks{Department of Mathematical Sciences, University of Mazandaran,
Babolsar, I. R. Iran {\tt n.besharati@umz.ac.ir}, Corresponding author},
J. B. Ebrahimi\thanks{Ecole Polytechnique Federale de Lausanne
(EPFL), Station 14, CH-1015 Lausanne, Switzerland {\tt
javad.ebrahimi@epfl.ch}},
A. Azadi
\thanks{Department of Mathematical Sciences, Sharif University of
Technology, P. O. Box 11155-9415, Tehran, I. R. Iran {\tt aazadi@gmail.com}}}
\date{}
\begin{document}

\maketitle
\begin{abstract}
Determining the size of a maximum independent set of a graph $G$,
denoted by $\alpha(G)$, is an NP-hard problem. Therefore many
attempts are made to find upper and lower bounds, or exact values
of $\alpha (G)$ for special classes of graphs.

This paper is aimed toward studying this problem for the class of
generalized Petersen graphs. We find new  upper and lower  bounds
and  some exact values for $\alpha(P(n,k))$. With a  computer
program  we have  obtained exact values for each  $n<78$. In
\cite{MR2381433} it is conjectured that
the size of the minimum vertex cover, $\beta(P(n, k))$,
 is less than or equal to
$n + \lceil\frac{n}{5}\rceil $, for all $n$ and $k$ with $n>2k$. We prove this
conjecture  for some cases. In particular, we show that if $ n>
3k$, the conjecture is valid. We checked the conjecture with our
table for $n < 78$ and it had no inconsistency. Finally, we show
that for every fixed $k$, $\alpha(P(n, k))$ can be computed using
an algorithm with running time $O(n)$.
\end{abstract}

\section {Introduction and preliminaries}

In a graph $G=(V,E)$, an {\sf independent set} $I(G)$ is a subset of the
vertices of $G$ such that no two vertices in $I(G)$ are adjacent.
The {\sf independence number} $\alpha (G)$ is the cardinality of a
largest set of independent vertices and an independent set of size $\alpha (G)$ is called an
$\alpha$-set.
The maximum independent set
problem is to find an independent set with the largest number of
vertices in a given graph. It is well-known that this problem is NP-hard \cite{MR519066}.
Therefore, many attempts are made to find  upper and lower bounds,
or exact values of $\alpha (G)$ for special classes of graphs.
This paper is aimed toward studying this problem for the generalized Petersen  graphs.

For each $n$ and $k$ \ $(n > 2k)$, a generalized Petersen graph
$P(n,k)$, is defined by vertex set $\{u_i, v_i\}$ and edge set
$\{u_iu_{i+1}, u_iv_i, v_iv_{i+k}\}$; where $i = 1, 2, \dots, n$
and subscripts are reduced modulo $n$. An induced subgraph on
$v$-{\sf vertices} is called the {\sf inner subgraph}, and
an induced subgraph on $u$-{\sf vertices} is called the {\sf outer cycle}.\\
In addition, we call two vertices $u_i$ and $v_i$ as twin of each other and the edge between them as a spoke.

In \cite{MR0038078}, Coxeter  introduced this class of graphs. Later
Watkins \cite{MR0236062} called these graphs  ``generalized Petersen graphs'', $P(n,k)$,
and conjectured that they admit a Tait coloring, except $P(5,2)$. This conjecture later
was proved in \cite{MR0304223}. Since 1969 this class of graphs
has been studied widely. Recently  vertex domination  and minimum vertex cover of $P(n,k)$
 have been studied.
For more details see for instance \cite{MR2381433}, \cite{MR2376483}, \cite{MR2521808} and \cite{MR2519173}.

A set $Q$ of vertices of a graph $G = (V,E)$ is called
a {\sf vertex cover} of $G$ if every edge of $G$ has at least one endpoint in $Q$. A vertex cover of a graph $G$
with the minimum possible cardinality is called a {\sf minimum vertex cover} of $G$ and its size is
denoted by $\beta(G)$.
In \cite{MR2381433} and \cite{MR2521808}
$\beta(P(n,k))$, has been studied. Since for every simple graph
$G$,   $\alpha(G) + \beta (G) = |V(G)|$ \cite{MR1367739}, their results imply the
following results for $\alpha(P(n,k))$, and  $n>2k$:
\begin{itemize}
\item [I)]
$\alpha(P(n,1)) = \left\{\begin{array}{lll}
 n   & & \hspace{0.5cm} n \hspace{0.5cm}is \hspace{0.5cm} even \\
 n-1 & & \hspace{0.5cm} n \hspace{0.5cm}is \hspace{0.5cm} odd.
 \end{array}\right.$
\item [II)]
For all $n > 4$, $\alpha(P(n,2))= \lfloor\frac{4n}{5}\rfloor$.
\item [III)]
For all $n > 6$,
$\alpha(P(n,3)) = \left\{\begin{array}{lll}
 n   & & \hspace{0.5cm}n \hspace{0.3cm}  is \hspace{0.3cm} even \\
 n-2 & & \hspace{0.5cm}n \hspace{0.3cm}  is \hspace{0.3cm} odd.
\end{array}\right.$
\item [IV)]
If both $n$ and $k$ are odd, then $\alpha (P(n,k))  \geq n- \frac{k+1}{2}$. \\
Also, if $ k \mid n$, then $\alpha(P(n,k)) = n- \frac{k+1}{2}$.
\item [V)]
$\alpha (P(n,k))$  $=n$ if and only if $n$ is even and $k$ is odd.
\item [VI)]
For all even $k$, we have
\begin{itemize}
\item
If $k-1\mid n$ then $\alpha(P(n, k))\geq n- \frac{n}{k-1}$.
\item
If $k-1 \nmid n$ then $\alpha(P(n, k))> n- \frac{n}{k-1}-2k$.
\end{itemize}
\item [VII)]
For all odd $n$, we have $\alpha(P(n, k)) \leq n -
\frac{d+1}{2}$, where $d= \gcd (n,k)$.
\end{itemize}
Recently, Fox et al.  proved the following results in
\cite{FoxRaluccaStanica}:
\begin{itemize}
\item[VIII)]
For all $n > 10$,
 $\alpha(P(n,5)) = \left\{\begin{array}{lll}
 n   & & \hspace{0.5cm} n \hspace{0.5cm}is \hspace{0.5cm} even \\
 n-3 & & \hspace{0.5cm} n \hspace{0.5cm}is \hspace{0.5cm} odd.
 \end{array}\right.$
\item[IX)]
For any integer $ k \geq 1$, we have that $\alpha(P(3k,k)) = \lceil{ \frac{5k-2}{2}}\rceil.$
\item[X)]
If $n,k$ are integers with $n$ odd and $k$ even, then
$\alpha (P(n,k))$ $\geq \frac{n-1}{2}+ \lfloor{\frac{\lceil{\frac{n}{k}}\rceil+1}{2}}\rfloor \cdot
\lfloor{\frac{n}{2d \lceil{\frac{n}{k}\rceil}}}\rfloor +\frac{d-1}{2}\cdot
\lfloor{\frac{1}{2}\big(\frac{n}{d} \ (mod \ \lceil{\frac{n}{k}}\rceil)\big)}\rfloor$,
where $d= \gcd (n,k)$.
\item[XI)]
If $n,k$ are even, then
$\alpha (P(n,k))$ $\geq \frac{n}{2}+ \frac{d}{2} \lfloor{\frac{n}{2d}}\rfloor$,  where $d= \gcd (n,k)$.
\end{itemize}
Notice that the problem of finding the size of a maximum
independent set in the graph $P(n,k)$ is trivial for even $n$ and
odd $k$, since $P(n,k)$ is a bipartite graph. For odd $n$ and $k$, $P(n,k)$
is not bipartite but we can remove a set of
 $k+1$ edges from $P(n,k)$ to obtain a bipartite graph. Thus in
 this case we have $ n-(k+1) \leq \alpha (P(n,k)) \leq n$.
 So, for odd $k$, we have upper and lower bounds for $\alpha (P(n,k))$ that are at most
$k+1$ away from $n$. In contrast, for even $k$, $P(n,k)$  has a lot
of odd cycles. In fact, the number of odd cycles in $P(n,k)$  is
at least as large as $O(n)$. This observation shows that for even
$k$, the graph $P(n,k)$ is far from being a bipartite graph and
as we see in continuation, we need more complicated arguments
for finding lower and upper bounds for $\alpha (P(n,k))$ compared to the case that $k$ is an odd number.

This paper is organized as follows. In Section $2$, we provide an
upper bound  for $\alpha (P(n,k))$ for  even $ k >2$. In Section $3$, we present some
lower bounds when   $k$ is even.
In both Section $2$ and Section $3$ we compare our   bounds  with
the  previously existing bounds. Some exact values for  $\alpha (P(n,k))$ are
given in Section $4$ by applying results presented in Sections $2$
and Section $3$. Finally, in Section
$5$ we prove  Behsaz-Hatami-Mahmoodian's conjecture  for some cases by using known
lower bounds. We checked the conjecture with our Table for $n < 78$, and it had no inconsistency.
\section{Upper bound}
In this section  we present an upper bound for $\alpha (P(n,k))$
when $k>2$ is even and we will show that the presented upper
bound is equal to $\alpha (P(n,k))$ in some cases. Our upper
bound is better than the upper bound given by Behsaz et.al. in
\cite{MR2381433}.\\
For $ t= 1,2, \dots,n$, we call the set
$\{u_t,u_{t+1},\dots,u_{t+2k-1}, v_t,v_{t+1},\dots,v_{t+2k-1}\}$ a $2k$-{\sf segment} and we denote it
by $I_t$. Let $G[I_t]$ be the subgraph of $P(n,k)$ induced by
$I_t$.

Let $\cal{S}$ be the set of all maximum independent sets of
$P(n,k)$. For every $S \in \cal{S}$ we denote by $f(S)$ the number of $2k$-segments $I_t$
for which $|I_t \cap S|= 2k \  (1 \leq t \leq n)$.\\
Define ${\cal{S}_{\min}} = \{ S \in {\cal{S}}| \forall S' \in {\cal{S}}, f(S) \leq f(S') \}$.
Since  $\cal{S}$ is nonempty,  ${\cal{S}_{\min}}$ is also  nonempty.
Let ${S_{0}}  \in \cal{S}_{\min}.$

\begin{proposition}
\label{Note}
For any $ S \in \cal{S}$, $ S \in \cal{S}_{\min}$ if and only if $f(S)=f(S_{0})$.
\end{proposition}

\begin{definition}
\label{type}
For any $S \in \cal{S}$, we say $I_t$ is of \  {\sf Type $1$} with respect to $S$
if $|I_t \cap S|= 2k$, of  \ {\sf Type $2$} with respect to $S$
if $|I_t \cap S|= 2k-1$, and of \  {\sf Type $3$} with respect to $S$ if $|I_t \cap S|\leq 2k-2$.\\
Let $T_{i}(S) = \{ I_t | I_t \  is  \ of \  Type \  i \  with  \ respect \  to \  S\}, \ for \  i =1,2,3$.
For a given $I_t \in T_2(S)$, we say $I_t$ is of
{\sf Special  type $2$} with respect to $S$,
if $u_t \notin S$ and $ \{u_t\} \cup (I_t \cap S)$ becomes an independent set for $G[I_t]$.
\end{definition}
Since $G[I_t]$ has a perfect matching of spokes
$\{u_tv_t,u_{t+1}v_{t+1},\dots,u_{2k+t-1}v_{2k+t-1}\}$, $|I_t \cap S|\leq 2k$.
So every $I_t$ is one  of the above Types.\\
Note that $f(S) = |T_1(S)|$.
\begin{lemma}
\label{Lemma type 1}
If $k$ is an even number then
$\alpha(G[I_t]) = 2k$ and $G[I_t]$ has a unique $\alpha$-set shown in Figure \ref{type 1}.\\
\begin{figure}[htb]
\begin{center}
\leavevmode
\includegraphics[width=0.8\textwidth]{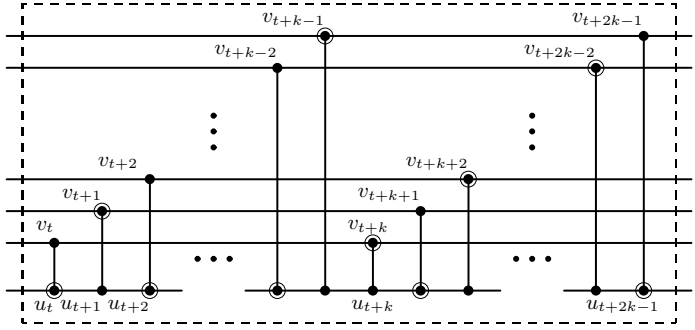}
\end{center}
\caption{$I_t$ as a Type $1$ segment.}
\label{type 1}
\end{figure}
\end{lemma}
\begin{proof}{
$G[I_t]$ has a perfect matching $\{u_iv_i |\  t\leq i \leq t+2k-1 \}$.
So $\alpha(G[I_t]) \leq  \frac{|V(G[I_t])|}{2} = 2k.$
On the other hand, Figure \ref{type 1} is an example of an independent set of $G[I_t]$ of size $2k$.
So $\alpha(G[I_t])= 2k$.\\
To show  that $G[I_t]$ has a unique  $\alpha$-set,
let $S$ be an $\alpha$-set of $(G[I_t])$.
Since $\alpha(G[I_t])= 2k$, $|S|= 2k $ and $S$ must contain precisely one vertex from each edge
$ \{u_i, v_i\}$ where $t \leq i \leq t+2k-1$.
Notice that the set of $u$-vertices of $(G[I_t])$ induces a path of length $2k$.
Therefore $|S \cap \{u_i | \ t \leq i \leq t+2k-1 \}| \leq k.$ The set of $v$-vertices
of $(G[I_t])$ induces a matching of
size $k$. This means that $|S \cap \{v_i | \  t \leq i \leq t+2k-1 \}| \leq k.$
These two observations show that any
$\alpha$-set $S$ of $G[I_t]$ has $k$ vertices from $u$-vertices and $k$ other vertices
from $v$-vertices of $V(G[I_t]).$
Moreover, every such $S$ contains precisely one vertex from each edge $ u_i v_i$
where $ t \leq i \leq t+2k-1$
and $v_i v_{i+k}$ where $ t \leq i \leq t+k-1$. Now, consider two cases:\\
Case 1: $v_t \in S$.\\
In this case, $u_t$ and $v_{t+k}$ are forced not to be in $S$. So $u_{t+k}$ is
forced to be in $S$. Then $u_{t+k-1}$
and $u_{t+k+1}$ are forced not to be in $S$ and this forces $v_{t+k-1}$ and $v_{t+k+1}$ to be in $S$.
Since $v_{t+k+1}$ is in $S$, $v_{t+1} \notin S$. Therefore $u_{t+1} \in S$, so $u_{t+2} \notin S$
and thus $v_{t+2} \in S.$
So, we showed that if $v_t \in S$ then $v_{t+2} \in S$ too. Now, if we repeat the same
argument for $v_{t+2}$ instead of
$v_t$, we can deduce that $v_{t+4} \in S$ and by a simple induction, it follows that $v_{t+2l} \in S$
for any $ 0 \leq l \leq \frac{k}{2} -1$. Particularly, $v_{t+k-2} \in S.$
Therefore $v_{t+2k-2} \notin S.$
This shows that $u_{t+2k-2} \in S$. Hence $u_{t+2k-1} \notin S$ and $v_{t+2k-1} \in S$.
So $v_{t+k-1} \notin S.$
But we already showed that $v_{t+k-1}$ is forced to be in $S$. This contradiction shows that there is
no Type $1$ $I_t$
for which $v_t \in S.$\\
Case 2: $u_t \in S$.\\
In this case, similar to the argument in  Case 1, each vertex is either forced to be in $S$ or it is forced not
to be in $S$. So, there is a unique pattern for $S \cap I_t$ when $I_t \in T_1(S)$.
Since the pattern shown in Figure \ref{type 1} is an instance of an independent set of size $2k$ for $G[I_t]$,
it is the unique pattern for such an independent set.}
~\end{proof}

Lemma \ref{Lemma type 1} guarantees that there is a unique pattern for $I_t \cap S$, if $I_t$
is of Special  type $2$ with respect to $S$.

\begin{lemma}
\label{Lemma uv}
For every $S \in \cal{S}$, if $I_t \in T_{1}(S)$
then $u_{2k+t},v_{2k+t}, u_{t-1}$, and $v_{t-1} \notin  S.$
Also, if $I_t $ is a  Special  type $2$ segment with respect
to $S$  then $u_{2k+t}$ and $v_{2k+t} \notin  S$.
\end{lemma}
\begin{proof}{
If $I_t \in T_{1}(S)$ then $|I_t \cap S|= 2k$. So by  Lemma \ref{Lemma type 1},
there is a unique pattern for $I_t \cap S$.
Based on this pattern, $u_{2k+t-1}$ and $v_{k+t} \in  S$. Therefore $u_{2k+t}$ and $v_{2k+t} \notin  S$,
since $S$ is an independent set of vertices of $P(n,k)$. A similar argument shows that $u_{t-1}$ and $v_{t-1} \notin  S$. The proof of the second part of the lemma is similar.}
\end{proof}
\begin{corollary}
\label{corollary 1}
If $I_t \in T_{1}(S)$ then $I_{t+1} , I_{t+2},\dots, I_{t+2k} \notin  T_{1}(S).$
\end{corollary}
\begin{proof}{
Notice that if $I_r \in  T_{1}(S)$ then for any edge $u_{i}v_{i} \in E(G[I_r])$ either
$u_i \in S$ or $v_i \in S$.
Since $I_t \in T_{1}(S)$, Lemma \ref{Lemma uv} implies that $u_{2k+t}$ and $v_{2k+t} \notin S$.
On the other hand,
 $u_{2k+t}v_{2k+t} \in E(G[I_{t+i}])$ for $ i = 1,2,\cdots,2k$.
 Thus $I_{t+1} , I_{t+2},\cdots,I_{t+2k} \notin  T_{1}(S).$}
\end{proof}
\begin{theorem}
\label{maintheorem}
$\alpha(P(n,k)) \leq \lfloor{\frac{(2k-1)n}{2k}}\rfloor$ for any even number
$k>2$ and any integer $ n > 2k$.
\end{theorem}
\begin{proof}{
Let $S_0 \in \cal{S}_{\min}$. We consider two cases.\\
Case $1$: $f(S_0) =0$.\\
In this case $T_1(S_0)= \emptyset$. So $|I_t \cap S_0|\leq 2k-1$
for any $ 1 \leq t \leq n$. If we add all of these $n$ inequalities we get:\\
\begin{equation}
\label{e: 1}
\sum_{t=1}^{n} |I_{t} \cap S_{0}|\leq (2k-1)n.
\end{equation}
On the other hand $\sum_{t=1}^{n} |I_{t} \cap S_{0}|= 2k|S_0|$,
since every element of $S_0$ is contained in precisely $2k$ of the sets $I_t$.
Thus:
$$2k|S_0| \leq (2k-1)n  \Longrightarrow \alpha(P(n,k)) = |S_0| \leq \frac{2k-1}{2k}n.$$
Case $2$: $f(S_0) > 0.$\\
In this case $T_1(S_0)\neq \emptyset $.
Similar to the inequality \ref{e: 1} we have:\\
$$2k|S_0|= \sum_{t=1}^{n} |I_{t} \cap S_{0}|\leq (2k-1)n +|T_{1}(S_0)| - |T_{3}(S_0)|.$$
So, to prove the theorem, it suffices to show that there exists $ S_0 \in \cal{S}_{\min}$ such that
$|T_{1}(S_0)|\leq |T_{3}(S_0)|$.\\
If we can show that for any $I_r \in T_{1}(S_0)$, there exists an
$I_{r'} \in T_{3}(S_0)$  so that $ I_{r+1},I_{r+2},\dots, I_{r'} \notin  T_{1}(S_0)$, then it follows that $|T_{1}(S_0)|\leq |T_{3}(S_0)|$.\\

On the  contrary, suppose that there exists $ I_{t} \in T_{1}(S_0)$ in such a way that in the sequence
$I_{t+1},I_{t+2},\dots$ before we see an element of $T_{3}(S_0)$, we see an element of $T_{1}(S_0)$.
Without loss of generality we can assume that $t=1.$ By  Lemma \ref{Lemma type 1},
 $(I_{1} \cap S_0)$ is of the form depicted in Figure \ref{type 11}.\\
\begin{figure}[htb]
\begin{center}
\leavevmode
\includegraphics[width=0.8\textwidth]{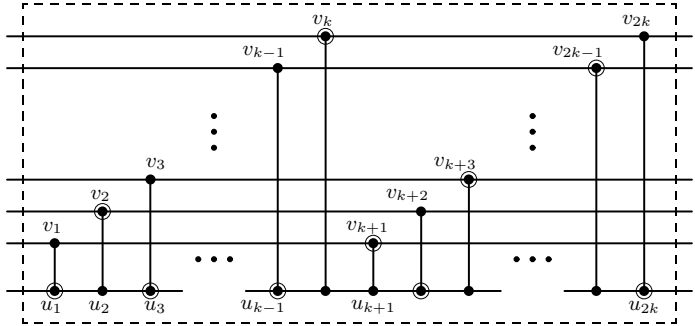}
\end{center}
\caption{$I_1$ as a Type $1$ segment.}
\label{type 11}
\end{figure}
Since $I_{1} \in T_{1}(S_0)$, by Corollary \ref{corollary 1},
$I_2,I_3, \dots, I_{2k+1} \notin  T_{1}(S_0).$
Based  on our assumption, $I_2,I_3, \dots, I_{2k+1} \in T_{2}(S_0).$
In particular,
$I_{2k+1} \in T_{2}(S_0)$.
Since $I_1 \in T_1(S_0)$, by Lemma \ref{Lemma uv} we have
$u_{2k+1}, v_{2k+1} \notin S_0$. On the other hand, we know that  $S_0$ must have
one vertex from each edge $u_iv_i$ where $2k+2 \leq i  \leq 4k.$
Since $2k+2 \leq 2k+3 \leq 4k$, either $u_{2k+3}$ or $v_{2k+3}
\in S_0$. But notice that $v_{2k+3}$ is adjacent to $v_{k+3}$
which is in $S_0$, for $k > 2$. Thus,  $v_{2k+3} \notin S_0$ and
$u_{2k+3}$ must be in $S_0$. This means that $u_{2k+2} \notin
S_0$. Now, define $S_1:= (S_0 \setminus \{u_{2k}\})\cup
\{u_{2k+1}\}$. One can easily see that $S_1 \in \cal{S}$.
Based on the choice of $S_0 \in \cal{S}_{\min}$, $f(S_0) \leq f(S_1)$.
Therefore, there must be an index
$2 \leq r  \leq n$ so that $I_r \in T_1(S_1)\setminus T_1(S_0)$.
Since $S_0$ and $S_1$ agree on every element except
$u_{2k}$ and $u_{2k+1}$, the only candidate for $r$ is $r= 2k+1$.
So $I_{2k+1} \in T_1(S_1)$ and $I_{2k+1} \notin T_1(S_0)$. Moreover
$I_{1} \in T_1(S_0)$ and $I_{1} \notin T_1(S_1)$. Thus $f(S_0) = f(S_1)$.
By Proposition \ref{Note}, $S_1 \in \cal{S}_{\min}$.
Notice that if any of $I_{2k+2},I_{2k+3},\dots,I_{n}$ are of Type $i$
with respect to $S_1$, they are of Type $i$ with respect to $S_0$, as well. So, in the sequence
$I_{2k+2},I_{2k+3},\dots,I_{n}$, any Type $3$ segment with respect to $S_1$
appears after an element of Type $1$ with respect to $S_1$.
 Since $I_{2k+1} \in T_1(S_1)$ by Corollary \ref{corollary 1},
$I_{2k+2},I_{2k+3},\dots,I_{4k+1} \notin T_1(S_1)$.
Then from our assumption $I_{2k+2},I_{2k+3},\dots,I_{4k+1} \in T_2(S_1)$.\\

This means that the same argument can be applied to $S_1$ and
if we define $S_2 := (S_1 \setminus \{u_{4k}\})\cup \{u_{4k+1}\}$, then
$S_2\in \cal{S}_{\min}$. If we consecutively repeat this argument for $S_1,S_2,S_3,\dots,S_m$ where $m=\lfloor{\frac{n}{2k}}\rfloor $ and
$S_i := (S_{i-1} \setminus \{u_{2ik}\}) \cup \{u_{2ki+1}\}$, then we  observe that
$S_i\in \cal{S}_{\min}$ and  $I_{2ki+1} \in T_1(S_i)$ for $ \ i= 1,2,\dots,m$,
and none of $I_2, I_3, \dots,I_{2k(m+1)}, I_{2k(m+1)+1}$
are of Type~$1$ with  respect to $S_0$. Also, $I_{2k(i+1)+1}$ for
$ \ i= 0,1,\dots,m$ are of Special  type~$2$ with  respect to $S_i$. Since
 $S_i$ and $S_0$ agree on the $I_{2ki+2}, I_{2ki+3}, \dots, I_{n}$, then $I_{2k(i+1)+1}$ for
$ \ i= 0,1,\dots,m$ are of Special  type~$2$ with  respect to $S_0$.\\
 In  other words, if $I_1$ belongs to $T_1(S_0)$ and the next element of $T_1(S_0)$
 appears before the first element of
$T_3(S_0)$ in the sequence $I_2,I_3,I_4,\dots,$ then all of $I_{2k+1},I_{4k+1},\dots,I_{2km+1}, I_{2k(m+1)+1}$
are Special type $2$ with respect to $S_0$.
In particular, $I_{2km+1}$ is of Special  type $2$ with  respect to $S_0$. \\
As $I_{2km+1} \in T_{1}(S_m)$, by Lemma \ref{Lemma uv},
$u_{2k(m+1)+1},v_{2k(m+1)+1}\notin S_m$ 
and since $S_m$ and $S_0$ agree on the
$I_{2km+2}, I_{2km+3}, \dots, I_{n}$ we conclude that
$u_{2k(m+1)+1}, v_{2k(m+1)+1}\notin S_0$.\\
Now consider three cases:
\begin{itemize}
\item{$ 2k(m+1)+1 \equiv 1\pmod{n}$:}\\
Since $I_{2km+1}$ is of Special  type $2$ with respect to $S_0$, by Lemma \ref{Lemma uv}, we have
$u_{2km+1+2k} = u_1 \notin S_0$. This is a contradiction as we assumed $I_1 \in T_1(S_0)$ and therefore
$u_1 \in S_0$.

\item{$ 2k(m+1)+1 \not \equiv 0,1\pmod{n}$:}\\
Since $I_1 \in T_1(S_0)$, by Lemma \ref{Lemma uv}, $u_n,v_n \notin S_0$. Also we know that $u_{2k(m+1)+1},v_{2k(m+1)+1}\notin S_0$.
Thus, $I_{2k(m+1)+1}$ is of Type $3$ with  respect to $S_0$
and none of $I_2,I_3,\dots,I_{2k(m+1)}$ are of Type $1$ with respect to $S_0$ which is a contradiction.
\item{$ 2k(m+1)+1  \equiv 0\pmod{n} $:}\\
$I_1$ is of Type $1$ with  respect to $S_0$ and for every $1\leq i \leq m $, $I_{2k(i+1)+1}$ is
of Special  type $2$ with  respect to $S_0$.
In particular, $I_{2km+1}$ is of Special  type $2$ with  respect to
$S_0$, and therefore
$v_{2km+k+2}=v_{2k(m+1)-k+3} \in S_0$, (See Figure~\ref{type 2}).
On the other hand,
$v_2 \in S_0$ as $I_1 \in T_1(S_0)$, and since $n=2k(m+1)+1$, $v_2$ is adjacent to $v_{2k(m+1)-k+3}$.
This is a contradiction.
\end{itemize}

So in all the cases, we get a contradiction  which means, after any Type $1$ segment
$I_r$, a Type $3$ segment $I_{r'}$ will appear before we see another Type $1$ segment. This means that
$|T_1(S_0)| \leq |T_3(S_0)|$ and the  theorem follows, as we argued earlier.}
\end{proof}
\section{Lower bounds}
In this section we introduce some lower bounds for $\alpha (P(n,k))$ where $k$ is even and $k > 2$.\\
Here we explain a construction for an independent set in $P(n,k)$ for even numbers  $n$ and $k$.
It happens that for every even $ n < 78$, our lower bound is equal to the actual value, using a computer
program for finding the independence number in $P(n,k)$.

\begin{theorem}
\label{nk even}
 If  $n$  and $k$ are even and $k>2$ then:
\begin{displaymath}
\alpha(P(n,k)) \geq  (2k-1)\lfloor{\frac{n}{2k}}\rfloor + \left\{ \begin{array}{ll}
\frac{r}{2}          &  \hspace{.5cm}\textrm{if $r \leq k$},\\  \\
\frac{3r}{2}-k-1     &  \hspace{.5cm}\textrm{if $r >k$}.
\end{array} \right.
\end{displaymath}
 where  $r$ is the remainder of $n$ modulo  $2k$.
\end{theorem}
\begin{proof}
{We partition the vertices of $P(n,k)$ into $\lfloor \frac{n}{2k}\rfloor$ $2k$-segments and one $r$-segment.
Since $n$ and $k$ are even numbers, $r$ is also an even number and it is straightforward to see that if we choose
a subset of the form shown in Figure \ref{type 2}, from each $2k$-segment they form an independent set $S_0$ of size
$(2k-1)\lfloor{\frac{n}{2k}}\rfloor$.

Then we try to extend this independent set by adding more vertices from the remaining
$r$-segment. Without loss of generality, we may assume that the $r$-segment consists of the vertices
$\{u_1, u_2, \dots, u_r,v_1,v_2, \dots, v_r\}$. consider two cases:
\begin{itemize}
\item { $r \leq k$ :}\\
In this case the set $S_0 \cup \{u_2, u_4, \dots, u_{r-2},u_r\}$ is an independent set of size
$(2k-1)\lfloor{\frac{n}{2k}}\rfloor + \frac{r}{2}$.
\item{ $r > k$ :}\\
In this case the set:\\
$ S_0 \cup \{u_3,u_5,\dots,u_{r-k-3}, u_{r-k-1}\}\  \cup
\{v_2,v_4,\dots,v_{r-k-2}, v_{r-k}\}\  \cup \\
\{u_{r-k+1},u_{r-k+3},\dots,u_{k-3}, u_{k-1}\} \
\cup \{u_{k+2},u_{k+4},\dots,u_{r-2}, u_{r}\}  \cup \\
\{v_{k+1},v_{k+3},\dots,v_{r-3}, v_{r-1}\}$
is an independent set of size
$(2k-1)\lfloor{\frac{n}{2k}}\rfloor + \frac{r-k-2}{2} + \frac{r-k}{2} + \frac{2k-r}{2}
+ \frac{r-k}{2} + \frac{r-k}{2} =  (2k-1)\lfloor{\frac{n}{2k}}\rfloor + \frac{3r}{2}-k-1.$
\end{itemize}}
~\end{proof}
%
\begin{figure}[htb]
\begin{center}
\leavevmode
\includegraphics[width=0.8\textwidth]{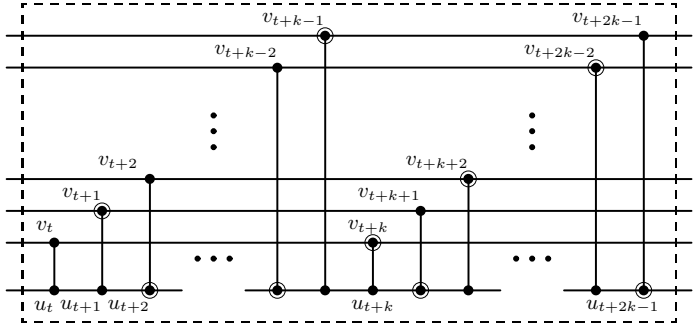}
\end{center}
\caption{$I_t$ as a Special type $2$ segment.}
\label{type 2}
\end{figure}
In the next theorem we establish a lower bound for $\alpha(P(n, k))$ for odd $n$ and even $k$.\\
\begin{theorem}
\label{n odd k even}
If $n$ is odd and $k>2$ is even then we have:
 \begin{center}
$ \alpha(P(n,k)) \geq (2k-1)\lfloor{\frac{n}{2k}}\rfloor
 + \left\{ \begin{array}{ll}
- \frac{k}{2} + 2            &  \hspace{.5cm}\textrm{if $r=1$},\\\\
\frac{3r-k-1}{2}             &  \hspace{.5cm}\textrm{if $ 1 < r < k $},\\\\
\frac{k}{2}+ \frac{r-1}{2}   &  \hspace{.5cm}\textrm{if $ k < r < 2k $}.\\
\end{array} \right.$
\end{center}
where  $r$ is the remainder of $n$ modulo  $2k$.
\end{theorem}
\begin{proof}{
We construct an independent set for the graph $P(n,k)$. Similar to the proof of Theorem \ref{nk even},
first we partition the vertices of the graph into $\lfloor\frac{n}{2k}\rfloor$ $2k$-segments
and a remaining segment of size $r$.
Without loss of generality, we can assume that the last $2k$-segment starts from the first spoke and
the remaining segment starts from the $(2k+1)$-st spoke and finishes at the $(2k+r)$-th spoke. We also label
the $2k$-segments with indices $1,2,\dots,\lfloor\frac{n}{2k}\rfloor$.

From each of $2k$-segments $1,2,\dots,\lfloor\frac{n}{2k}\rfloor -1 $, we choose $2k-1$ vertices as
shown in Figure \ref{type 2}. We also choose the following
vertices from the last $2k$-segment and the remaining $r$-segment:
\begin{itemize}
\item{$r=1$:}\\
$\{u_2, u_4, \dots, u_{k-2},u_k\}\  \cup \{u_{k+3}, u_{k+5}, \dots, u_{2k-1},u_{2k+1}\} \ \cup
\{v_{k+1}\} \ \cup \\ \{v_{k+2},v_{k+4},\dots,v_{2k-2}, v_{2k}\}.$
\item {$1 < r < k: $}\\
$\{u_3,u_5,\dots,u_{k-3}, u_{k-1}\}  \ \cup \ \{u_{k+2}, u_{k+4}, \dots, u_{2k-2},u_{2k}\} \ \cup \\
\{u_{2k+3},u_{2k+5},\dots,u_{2k+r-2}, u_{2k+r}\} \ \cup \  \{v_2,v_4,\dots,v_{k-2}, v_{k}\} \  \cup  \\
\{v_{k+1},v_{k+3},\dots,v_{k+r-2}, v_{k+r}\}  \ \cup \
\{v_{2k+2},v_{2k+4},\dots,v_{2k+r-3}, v_{2k+r-1}\}. $
\item {$k < r < 2k: $}\\
$\{u_3,u_5,\dots,u_{k-3}, u_{k-1}\} \  \cup \  \{u_{k+2}, u_{k+4}, \dots, u_{2k-2},u_{2k}\}  \ \cup \\
\{u_{2k+3},u_{2k+5},\dots,u_{2k+r-2}, u_{2k+r}\} \ \cup \ \{v_2,v_4,\dots,v_{k-2}, v_{k}\}  \ \cup \\
\{v_{k+1},v_{k+3},\dots,v_{2k-3}, v_{2k-1}\}  \ \cup \ \{v_{2k+2},v_{2k+4},\dots,v_{3k-2}, v_{3k}\}. $
\end{itemize}
One can easily check that in each case, the given set is an independent set of size specified in the theorem.}
~\end{proof}
Notice that the upper bound given in Theorem \ref{maintheorem} and the lower bound in
Theorem \ref{nk even}, and Theorem \ref{n odd k even} are very close to each other for every fixed even $k > 2$. More precisely, we have the following corollary:
\begin{corollary}
If $ k>2$ is an even number then $\alpha(P(n,k))= \frac{(2k-1)}{2k}n + O(k).$
\end{corollary}

Notice that our lower bounds are considerably better than the lower
bounds obtained  in \cite{MR2381433} and \cite{FoxRaluccaStanica}.
\section{Some exact values}
In this section, we will find the exact value of  $\alpha (P(n,k))$ for some pairs of $n,k$.\\

\begin{proposition}
\label{k=4}
If $n > 8$, then:\\
\begin{displaymath}
\alpha(P(n,4))=
\left\{ \begin{array}{ll}
\frac{7n}{8}          &  \hspace{.5cm}\textrm{if $n \equiv 0 \pmod{8}$},\\  \\
\frac{7}{8}(n-1)      &  \hspace{.5cm}\textrm{if $n \equiv 1 \pmod{8}$},\\  \\
\frac{7}{8}(n-2) +1   &  \hspace{.5cm}\textrm{if $n \equiv 2 \pmod{8}$},\\ \\
\frac{7}{8}(n-3) +2   &  \hspace{.5cm}\textrm{if $n \equiv 3 \pmod{8}$},\\ \\
\frac{7}{8}(n-5) +4   &  \hspace{.5cm}\textrm{if $n \equiv 5 \pmod{8}$}.\\
\end{array} \right.
\end{displaymath}

\begin{displaymath}
\alpha(P(n,4))\geq
\left\{ \begin{array}{ll}
\frac{7}{8}(n-4)+2   &  \hspace{.5cm}\textrm{if $n \equiv 4 \pmod{8}$},\\ \\
\frac{7}{8}(n-6)+4   &  \hspace{.5cm}\textrm{if $n \equiv 6 \pmod{8}$},\\\\
\frac{7}{8}(n-7)+5   &  \hspace{.5cm}\textrm{if $n \equiv 7 \pmod{8}$}.\\
\end{array} \right.
\end{displaymath}
\end{proposition}
\begin{proof}
{This result is straight consequence of Theorems
~\ref{maintheorem}, \ref{nk even}, and \ref{n odd k even}.}
~\end{proof}
Notice that for $k=4$ and $ n \equiv 4,6$ or   $7 \pmod 8$, the upper bound and lower bound differ
by $1$. In fact, for $ n<700$ the exact of $\alpha(P(n,k))$ is the same as our lower bound as
we checked by computer.

\begin{proposition}
If  $ k>2$ is an even number and $n \equiv 0,  2, k-1$ or $k+1 \pmod{2k}$ then
$\alpha(P(n,k))= \lfloor{\frac{(2k-1)n}{2k}}\rfloor$.
\end{proposition}
\begin{proof}
{This assertion is trivial consequence of Theorems
\ref{maintheorem},\ref{nk even} and \ref{n odd k even}. In fact
the upper bound and lower bounds we have for $\alpha(P(n,k))$ are
identical in these cases.} ~\end{proof}
\section{Behsaz-Hatami-Mahmoodian's conjecture}
\begin{conjecture}
\label{conj 1}
{\rm(\cite{MR2381433}).}
For all $n$, $k$ we have $\beta (P(n,k)) \leq n + \lceil \frac{n}{5}\rceil$.\\
\end{conjecture}
Notice that, since $\alpha(G)+ \beta(G)= |V(G)|$, this conjecture is equivalent to
$\alpha(P(n,k))\geq \lfloor{\frac{4n}{5}}\rfloor$.

\begin{theorem}
The above conjecture is valid in the following cases:
\begin{itemize}
\item [a)] $k=1,2,3,4,5$.
\item [b)] $n$ is even and $k$ is odd.
\item [c)] $n,k$ are odd and $ n \geq \frac{5(k+1)}{2}$.
\item [d)] $n,k$ are even.
\item [e)] $n$ is odd,  $k$ is even  and $ n > 3k$.
\end{itemize}
\end{theorem}
\begin{proof}{
\begin{itemize}
\item [a)]
This case is a straight consequence of (I), (II), (III),
 Proposition \ref{k=4} and (VIII).
\item [b)]
In this case $P(n,k)$ is a bipartite graph and $\alpha(P(n,k)) = n$.
\item [c)]
$\alpha(P(n,k)) \geq n- \frac{k+1}{2} $ {\rm(\cite{MR2381433})}.
For $n \geq \frac{5(k+1)}{2}$ this  lower bound
is greater than  $\lfloor\frac{4n}{5}\rfloor$.
\item [d)]
Let $n=2kq+r$ where $q\geq1$ and $0 \leq r < 2k$. We consider the following subcases:
\begin{itemize}
\item
If $r \leq k$  and $q=1$ then by Theorem \ref{nk even},
$\alpha(P(n,k)) \geq 2k-1 +\frac{r}{2} \geq {\frac{4}{5}}(2k+r)$ for any $ k \geq 10$. For $ k < 10$
conjecture holds based on the information provided in Table $1$.
\item
If $ r \leq k$ and $q > 1$  then by Theorem \ref{nk even},
$\alpha(P(n,k)) \geq (2k-1)q +\frac{r}{2} \geq {\frac{4}{5}}(2kq+r)$
for any $ k \geq 4$. For $ k < 4$, conjecture follows from $(a)$.
\item
If $r > k$  and $q=1$ then by Theorem \ref{nk even},
$\alpha(P(n,k)) \geq 2k-1 +\frac{3r}{2}-k-1 \geq {\frac{4}{5}}(2k+r)$ for any $ k \geq 20$.
For $ k < 20$ conjecture holds based on the information provided in Table $1$.
\item
If $ r > k$ and $q > 1$  then by Theorem \ref{nk even},
$\alpha(P(n,k)) \geq (2k-1)q + \frac{3r}{2}-k-1\geq {\frac{4}{5}}(2kq+r)$
for any $ k \geq 6$. For $ k < 6$, conjecture follows from $(a)$.
\end{itemize}
\item [e)]
Similar to the previous part, let $n=2kq+r$ where $q\geq1$  and
$0 \leq r < 2k$. We consider the following subcases:
\begin{itemize}
\item
If $r =1$  and $q=1$ then $P(n,k)= P(2k+1,k)$ which is isomorphic to $P(2k+1,2)$.
(For more information about isomorphic generalized Petersen graphs see \cite{MR2475014}).
\item
If $r =1$  and $q \geq 2$ then  by Theorem \ref{n odd k even},
$\alpha(P(n,k)) \geq (2k-1)q - \frac{k}{2}+2 \geq {\frac{4}{5}}(2kq+1)$
for every $k >2$.
\item
If $ 1<r<k$ then $q$ has to be larger than $1$. In fact if $ 1<r<k$ and
$q=1$ then $n=2kq+r < 3k$.
For $ 1<r<k$ and $ q \geq 2$  then by Theorem \ref{n odd k even},
$\alpha(P(n,k)) \geq (2k-1)q + \frac{3r-k-1}{2} \geq {\frac{4}{5}}(2kq+r)$ for
every $k \geq 2$. (Note that since $n$ is odd and $k$ is even, $ r >1$ implies that $r\geq 3$).
\item
If $ k < r < 2k$ then  by Theorem \ref{n odd k even},
$\alpha(P(n,k)) \geq (2k-1)q +  \frac{k}{2}+ \frac{r-1}{2} \geq {\frac{4}{5}}(2kq+r)$
for $ k \geq 5$. For $ k < 5$ then the assertion is concluded from part $(a)$.
\end{itemize}
\end{itemize}}
~\end{proof}
\begin{corollary}
If $ n > 3k$ then $\alpha(P(n,k)) \geq \lfloor\frac{4n}{5}\rfloor$, and
Behsaz-Hatami-Mahmoodian's conjecture holds.
\end{corollary}
\section{Polynomial algorithm for $\alpha(P(n,k))$}
In this section we will prove that the independence number of generalized
Petersen graphs with fixed $k$ can be found in linear time, $O(n)$.
This result is a special case of a deep theorem stating that the problem of finding the
independence number of graphs with bounded treewidth can be solved in linear time of
the number of vertices of the graph.

In the continuation we will show that for every
fixed $k$ and any integer $n>2k$ the treewidth of $P(n,k)$ is bounded.
First, we need to formally define the concepts of tree decomposition and treewidth of a graph.
\begin{definition}
Let $G=(V,E)$ be a graph. A tree decomposition of $G$ is a pair $(X, T)$, where
$X = \{X_1, X_2, \ldots ,X_n \} $ is a family of subsets of $V$, and $T$ is a tree whose nodes are
the subsets $X_j$, satisfying the following properties:
\begin{itemize}
\item[1)]
The union of all sets $X_j$ equals $V$.
\item[2)]
For every edge $(v,w)$ in the graph, there is a subset $X_j$ that contains both $v$ and $w$.
\item[3)]
If $X_j$ is on the path from $X_i$ to $X_l$ in $T$ then $X_i \cap X_l \subseteq X_j$.
In  other words, for all vertices $v \in V,$ all nodes $X_i$ which contain $v$ induce
a connected subtree of $T$.
\end{itemize}
The width of $(X,T)$ is defined to be the size of the largest $X_i$ minus one.
The treewidth, $tw(G)$, of the
graph $G$ is defined to be the minimum width of all its tree decompositions.
The treewidth  will be taken as a measure of how much a graph resembles
a tree.
\end{definition}
\begin{theorema}
{\rm(\cite{MR1268488})}
\label{alghorithm}
The problem of finding a maximum independent set of a graph $G$ with bounded treewidth,  $tw(G)\leq l$
can be solved in $O(2^{3l}n)$  by dynamic programming
techniques, where $n$ is the number of  vertices of  graph.
\end{theorema}
For more details see for instance
\cite{MR985145}, \cite{MR1268488}, \cite{MR1023630}, 
and  \cite{MR855559}.
\begin{theorem}
For any fixed $k$, the problem of finding independence number of the graphs $P(n,k)$ can be
solved by an algorithm with running time $O(n)$.
\end{theorem}

\begin{proof}{
By  Theorem~\ref{alghorithm}, we only need to show that for a
given number $k$, the treewidth of $P(n,k)$ is bounded.
Consider the following tree decomposition of $P(n,k)$ of
width $4k+3$. Without loss of generality we can only consider the case where $n>2k+1$.
Let $T$ be the path of order $n-2k-1$ and define
$X_1, X_2, \dots, X_{n-2k-1}$ as follows:\\
$X_1= \{u_1, v_1, u_2, v_2, \dots, u_{k+1},v_{k+1}, u_{n-k},v_{n-k}, u_{n-k+1}, v_{n-k+1}, \dots, u_n,v_n\},$\\
$X_2=  (X_1 \setminus \{u_1, v_1 \}) \cup  \{u_{k+2}, v_{k+2} \},$\\
$X_3=  (X_2 \setminus \{u_n, v_n \}) \cup  \{u_{n-k-1}, v_{n-k-1}\},$\\
$X_4=  (X_3 \setminus \{u_2, v_2 \}) \cup  \{u_{k+3}, v_{k+3}\},$ \\
$X_5=  (X_4 \setminus \{u_{n-1}, v_{n-1} \}) \cup  \{u_{n-k-2}, v_{n-k-2} \},$
and so on.\\
Notice that
in each step we remove
two elements and add two other elements. Therefore
$|X_i|= 4k+4$ for all $i$.  One can easily see that $(X,T)$ is a tree decomposition for $P(n,k)$
where $X=\{X_1, X_2, \dots, X_{n-2k-1}\}$. Thus,  $ tw(P(n,k)) \leq 4k+3$ and by Theorem \ref{alghorithm},
the proof is complete.} ~\end{proof}
\tiny{
\def\arraystretch{1.05}
\begin{tabular}
{|@{\hspace{1pt}}c@{\hspace{1pt}}||@{\hspace{1pt}}c@{\hspace{1pt}}|@{\hspace{1pt}}c@{\hspace{1pt}}|@{\hspace{1pt}}c@{\hspace{1pt}}|@{\hspace{1pt}}c@
{\hspace{1pt}}|@{\hspace{1pt}}c@{\hspace{1pt}}|@{\hspace{1pt}}c@{\hspace{1pt}}|@{\hspace{1pt}}c@{\hspace{1pt}}|@{\hspace{1pt}}c@
{\hspace{1pt}}|@{\hspace{1pt}}c@{\hspace{1pt}}|@{\hspace{1pt}}c@{\hspace{1pt}}|@{\hspace{1pt}}c@{\hspace{1pt}}|@{\hspace{1pt}}c@
{\hspace{1pt}}|@{\hspace{1pt}}c@{\hspace{1pt}}|@{\hspace{1pt}}c@{\hspace{1pt}}|@{\hspace{1pt}}c@{\hspace{1pt}}|@{\hspace{1pt}}c@
{\hspace{1pt}}|@{\hspace{1pt}}c@{\hspace{1pt}}|@{\hspace{1pt}}c@{\hspace{1pt}}|@{\hspace{1pt}}c@{\hspace{1pt}}|@{\hspace{1pt}}c@
{\hspace{1pt}}|@{\hspace{1pt}}c@{\hspace{1pt}}|@{\hspace{1pt}}c@{\hspace{1pt}}|@{\hspace{1pt}}c@{\hspace{1pt}}|@{\hspace{1pt}}c@
{\hspace{1pt}}|@{\hspace{1pt}}c@{\hspace{1pt}}|@{\hspace{1pt}}c@{\hspace{1pt}}|@{\hspace{1pt}}c@{\hspace{1pt}}|@{\hspace{1pt}}c@
{\hspace{1pt}}|@{\hspace{1pt}}c@{\hspace{1pt}}|@{\hspace{1pt}}c@{\hspace{1pt}}|@{\hspace{1pt}}c@{\hspace{1pt}}|@{\hspace{1pt}}c@
{\hspace{1pt}}|@{\hspace{1pt}}c@{\hspace{1pt}}|@{\hspace{1pt}}c@{\hspace{1pt}}|@{\hspace{1pt}}c@{\hspace{1pt}}|@{\hspace{1pt}}c@
{\hspace{1pt}}|@{\hspace{1pt}}c@{\hspace{1pt}}|@{\hspace{1pt}}c@{\hspace{1pt}}|}
\hline
\backslashbox[0pt][lr]{$n$}{$k$}&  1 & 2 & 3 & 4 & 5 & 6  & 7 & 8 & 9 & 10 & 11 & 12 & 13 & 14 & 15 & 16 & 17 & 18 & 19 & 20 & 21 & 22 & 23 & 24 & 25 & 26 & 27 & 28 & 29 & 30 & 31 & 32 & 33 & 34 & 35 & 36 & 37 & 38 \\ \hline
5  & 4  & 4  &    &    &    &    &    &    &    &    &    &    &    &    &    &    &    &    &    &    &    &    &    &    &    &    &    &    &    &    &    &    &    &    &    &    &    &  \\ \hline
6  & 6  & 4  &    &    &    &    &    &    &    &    &    &    &    &    &    &    &    &    &    &    &    &    &    &    &    &    &    &    &    &    &    &    &    &    &    &    &    &  \\ \hline
7  & 6  & 5  & 5  &    &    &    &    &    &    &    &    &    &    &    &    &    &    &    &    &    &    &    &    &    &    &    &    &    &    &    &    &    &    &    &    &    &    & \\ \hline
8  & 8  & 6  & 8  &    &    &    &    &    &    &    &    &    &    &    &    &    &    &    &    &    &    &    &    &    &    &    &    &    &    &    &    &    &    &    &    &    &    &\\ \hline
9  & 8  & 7  & 7  & 7  &    &    &    &    &    &    &    &    &    &    &    &    &    &    &    &    &    &    &    &    &    &    &    &    &    &    &    &    &    &    &    &    &    & \\ \hline
10 & 10 & 8  & 10 & 8  &    &    &    &    &    &    &    &    &    &    &    &    &    &    &    &    &    &    &    &    &    &    &    &    &    &    &    &    &    &    &    &    &    & \\ \hline
11 & 10 & 8  & 9  & 9  & 8  &    &    &    &    &    &    &    &    &    &    &    &    &    &    &    &    &    &    &    &    &    &    &    &    &    &    &    &    &    &    &    &    & \\ \hline
12 & 12 & 9  & 12 & 9  & 12 &    &    &    &    &    &    &    &    &    &    &    &    &    &    &    &    &    &    &    &    &    &    &    &    &    &    &    &    &    &    &    &    & \\ \hline
13 & 12 & 10 & 11 & 11 & 10 & 10 &    &    &    &    &    &    &    &    &    &    &    &    &    &    &    &    &    &    &    &    &    &    &    &    &    &    &    &    &    &    &    & \\ \hline
14 & 14 & 11 & 14 & 11 & 14 & 12 &    &    &    &    &    &    &    &    &    &    &    &    &    &    &    &    &    &    &    &    &    &    &    &    &    &    &    &    &    &    &    & \\ \hline
15 & 14 & 12 & 13 & 12 & 12 & 12 & 12 &    &    &    &    &    &    &    &    &    &    &    &    &    &    &    &    &    &    &    &    &    &    &    &    &    &    &    &    &    &    & \\ \hline
16 & 16 & 12 & 16 & 14 & 16 & 13 & 16 &    &    &    &    &    &    &    &    &    &    &    &    &    &    &    &    &    &    &    &    &    &    &    &    &    &    &    &    &    &    & \\ \hline
17 & 16 & 13 & 15 & 14 & 14 & 15 & 14 & 13 &    &    &    &    &    &    &    &    &    &    &    &    &    &    &    &    &    &    &    &    &    &    &    &    &    &    &    &    &    &\\ \hline
18 & 18 & 14 & 18 & 15 & 18 & 14 & 18 & 16 &    &    &    &    &    &    &    &    &    &    &    &    &    &    &    &    &    &    &    &    &    &    &    &    &    &    &    &    &    & \\ \hline
19 & 18 & 15 & 17 & 16 & 16 & 17 & 15 & 15 & 15 &    &    &    &    &    &    &    &    &    &    &    &    &    &    &    &    &    &    &    &    &    &    &    &    &    &    &    &    & \\ \hline
20 & 20 & 16 & 20 & 16 & 20 & 16 & 20 & 17 & 20 &    &    &    &    &    &    &    &    &    &    &    &    &    &    &    &    &    &    &    &    &    &    &    &    &    &    &    &    & \\ \hline
21 & 20 & 16 & 19 & 18 & 18 & 18 & 17 & 18 & 17 & 16 &    &    &    &    &    &    &    &    &    &    &    &    &    &    &    &    &    &    &    &    &    &    &    &    &    &    &    & \\ \hline
22 & 22 & 17 & 22 & 18 & 22 & 19 & 22 & 18 & 22 & 20 &    &    &    &    &    &    &    &    &    &    &    &    &    &    &    &    &    &    &    &    &    &    &    &    &    &    &    & \\ \hline
23 & 22 & 18 & 21 & 19 & 20 & 19 & 19 & 21 & 20 & 19 & 18 &    &    &    &    &    &    &    &    &    &    &    &    &    &    &    &    &    &    &    &    &    &    &    &    &    &    & \\ \hline
24 & 24 & 19 & 24 & 21 & 24 & 22 & 24 & 19 & 24 & 21 & 24 &    &    &    &    &    &    &    &    &    &    &    &    &    &    &    &    &    &    &    &    &    &    &    &    &    &    & \\ \hline
25 & 24 & 20 & 23 & 21 & 22 & 21 & 21 & 23 & 20 & 21 & 20 & 20 &    &    &    &    &    &    &    &    &    &    &    &    &    &    &    &    &    &    &    &    &    &    &    &    &    & \\ \hline
26 & 26 & 20 & 26 & 22 & 26 & 23 & 26 & 21 & 26 & 22 & 26 & 24 &    &    &    &    &    &    &    &    &    &    &    &    &    &    &    &    &    &    &    &    &    &    &    &    &    & \\ \hline
27 & 26 & 21 & 25 & 23 & 24 & 23 & 23 & 24 & 22 & 24 & 24 & 22 & 21 &    &    &    &    &    &    &    &    &    &    &    &    &    &    &    &    &    &    &    &    &    &    &    &    & \\ \hline
28 & 28 & 22 & 28 & 23 & 28 & 24 & 28 & 24 & 28 & 23 & 28 & 25 & 28 &    &    &    &    &    &    &    &    &    &    &    &    &    &    &    &    &    &    &    &    &    &    &    &    & \\ \hline
29 & 28 & 23 & 27 & 25 & 26 & 26 & 25 & 25 & 24 & 27 & 25 & 24 & 24 & 23 &    &    &    &    &    &    &    &    &    &    &    &    &    &    &    &    &    &    &    &    &    &    &    & \\ \hline
30 & 30 & 24 & 30 & 25 & 30 & 25 & 30 & 27 & 30 & 24 & 30 & 26 & 30 & 28 &    &    &    &    &    &    &    &    &    &    &    &    &    &    &    &    &    &    &    &    &    &    &    & \\ \hline
31 & 30 & 24 & 29 & 26 & 28 & 28 & 27 & 26 & 27 & 29 & 25 & 27 & 27 & 25 & 24 &    &    &    &    &    &    &    &    &    &    &    &    &    &    &    &    &    &    &    &    &    &    & \\ \hline
32 & 32 & 25 & 32 & 28 & 32 & 27 & 32 & 30 & 32 & 26 & 32 & 28 & 32 & 29 & 32 &    &    &    &    &    &    &    &    &    &    &    &    &    &    &    &    &    &    &    &    &    &    & \\ \hline
33 & 32 & 26 & 31 & 28 & 30 & 29 & 29 & 28 & 28 & 30 & 27 & 30 & 30 & 29 & 27 & 26 &    &    &    &    &    &    &    &    &    &    &    &    &    &    &    &    &    &    &    &    &    & \\ \hline
34 & 34 & 27 & 34 & 29 & 34 & 30 & 34 & 31 & 34 & 29 & 34 & 28 & 34 & 30 & 34 & 32 &    &    &    &    &    &    &    &    &    &    &    &    &    &    &    &    &    &    &    &    &    & \\ \hline
35 & 34 & 28 & 33 & 30 & 32 & 30 & 31 & 30 & 30 & 31 & 29 & 33 & 30 & 30 & 30 & 29 & 28 &    &    &    &    &    &    &    &    &    &    &    &    &    &    &    &    &    &    &    &    & \\ \hline
36 & 36 & 28 & 36 & 30 & 36 & 33 & 36 & 32 & 36 & 32 & 36 & 29 & 36 & 31 & 36 & 33 & 36 &    &    &    &    &    &    &    &    &    &    &    &    &    &    &    &    &    &    &    &    & \\ \hline
37 & 36 & 29 & 35 & 32 & 34 & 32 & 33 & 33 & 32 & 32 & 32 & 35 & 30 & 33 & 34 & 33 & 30 & 29 &    &    &    &    &    &    &    &    &    &    &    &    &    &    &    &    &    &    &    & \\ \hline
38 & 38 & 30 & 38 & 32 & 38 & 34 & 38 & 33 & 38 & 35 & 38 & 31 & 38 & 33 & 38 & 34 & 38 & 36 &    &    &    &    &    &    &    &    &    &    &    &    &    &    &    &    &    &    &    & \\ \hline
39 & 38 & 31 & 37 & 33 & 36 & 34 & 35 & 36 & 34 & 33 & 35 & 36 & 32 & 36 & 35 & 33 & 33 & 32 & 31 &    &    &    &    &    &    &    &    &    &    &    &    &    &    &    &    &    &    & \\ \hline
40 & 40 & 32 & 40 & 35 & 40 & 35 & 40 & 34 & 40 & 38 & 40 & 35 & 40 & 33 & 40 & 35 & 40 & 37 & 40 &    &    &    &    &    &    &    &    &    &    &    &    &    &    &    &    &    &    & \\ \hline
41 & 40 & 32 & 39 & 35 & 38 & 37 & 37 & 38 & 36 & 35 & 35 & 37 & 34 & 39 & 35 & 36 & 37 & 36 & 34 & 32 &    &    &    &    &    &    &    &    &    &    &    &    &    &    &    &    &    & \\ \hline
42 & 42 & 33 & 42 & 36 & 42 & 36 & 42 & 36 & 42 & 39 & 42 & 37 & 42 & 34 & 42 & 37 & 42 & 38 & 42 & 40 &    &    &    &    &    &    &    &    &    &    &    &    &    &    &    &    &    & \\ \hline
43 & 42 & 34 & 41 & 37 & 40 & 39 & 39 & 39 & 38 & 37 & 37 & 38 & 37 & 41 & 35 & 39 & 40 & 38 & 38 & 35 & 34 &    &    &    &    &    &    &    &    &    &    &    &    &    &    &    &    & \\ \hline
44 & 44 & 35 & 44 & 37 & 44 & 38 & 44 & 39 & 44 & 40 & 44 & 40 & 44 & 36 & 44 & 38 & 44 & 39 & 44 & 41 & 44 &    &    &    &    &    &    &    &    &    &    &    &    &    &    &    &    & \\ \hline
45 & 44 & 36 & 43 & 39 & 42 & 40 & 41 & 40 & 40 & 40 & 39 & 39 & 41 & 42 & 37 & 42 & 40 & 39 & 40 & 39 & 37 & 36 &    &    &    &    &    &    &    &    &    &    &    &    &    &    &    & \\ \hline
46 & 46 & 36 & 46 & 39 & 46 & 41 & 46 & 42 & 46 & 41 & 46 & 43 & 46 & 40 & 46 & 38 & 46 & 40 & 46 & 42 & 46 & 44 &    &    &    &    &    &    &    &    &    &    &    &    &    &    &    & \\ \hline
47 & 46 & 37 & 45 & 40 & 44 & 41 & 43 & 41 & 42 & 43 & 41 & 40 & 42 & 43 & 39 & 45 & 41 & 42 & 44 & 43 & 42 & 39 & 37 &    &    &    &    &    &    &    &    &    &    &    &    &    &    & \\ \hline
48 & 48 & 38 & 48 & 42 & 48 & 44 & 48 & 45 & 48 & 42 & 48 & 46 & 48 & 42 & 48 & 39 & 48 & 44 & 48 & 43 & 48 & 45 & 48 &    &    &    &    &    &    &    &    &    &    &    &    &    &    & \\ \hline
49 & 48 & 39 & 47 & 42 & 46 & 43 & 45 & 43 & 44 & 46 & 44 & 42 & 42 & 44 & 42 & 47 & 40 & 45 & 45 & 42 & 43 & 42 & 40 & 39 &    &    &    &    &    &    &    &    &    &    &    &    &    & \\ \hline
50 & 50 & 40 & 50 & 43 & 50 & 45 & 50 & 46 & 50 & 43 & 50 & 47 & 50 & 45 & 50 & 41 & 50 & 43 & 50 & 44 & 50 & 46 & 50 & 48 &    &    &    &    &    &    &    &    &    &    &    &    &    & \\ \hline
51 & 50 & 40 & 49 & 44 & 48 & 45 & 47 & 45 & 46 & 48 & 45 & 44 & 44 & 45 & 46 & 48 & 42 & 48 & 45 & 45 & 47 & 47 & 45 & 42 & 40 &    &    &    &    &    &    &    &    &    &    &    &    & \\ \hline
52 & 52 & 41 & 52 & 44 & 52 & 46 & 52 & 47 & 52 & 45 & 52 & 48 & 52 & 48 & 52 & 45 & 52 & 43 & 52 & 46 & 52 & 47 & 52 & 49 & 52 &    &    &    &    &    &    &    &    &    &    &    &    & \\ \hline
53 & 52 & 42 & 51 & 46 & 50 & 48 & 49 & 48 & 48 & 49 & 47 & 47 & 46 & 46 & 49 & 49 & 44 & 51 & 46 & 48 & 50 & 47 & 46 & 47 & 44 & 42 &    &    &    &    &    &    &    &    &    &    &    & \\ \hline
54 & 54 & 43 & 54 & 46 & 54 & 47 & 54 & 48 & 54 & 48 & 54 & 49 & 54 & 51 & 54 & 48 & 54 & 44 & 54 & 49 & 54 & 48 & 54 & 50 & 54 & 52 &    &    &    &    &    &    &    &    &    &    &    & \\ \hline
55 & 54 & 44 & 53 & 47 & 52 & 50 & 51 & 51 & 50 & 50 & 49 & 50 & 48 & 47 & 49 & 50 & 48 & 53 & 45 & 51 & 50 & 48 & 50 & 50 & 48 & 45 & 44 &    &    &    &    &    &    &    &    &    &    & \\ \hline
56 & 56 & 44 & 56 & 49 & 56 & 49 & 56 & 49 & 56 & 51 & 56 & 50 & 56 & 54 & 56 & 50 & 56 & 46 & 56 & 49 & 56 & 49 & 56 & 51 & 56 & 53 & 56 &    &    &    &    &    &    &    &    &    &    & \\ \hline
57 & 56 & 45 & 55 & 49 & 54 & 51 & 53 & 53 & 52 & 51 & 51 & 53 & 51 & 49 & 49 & 51 & 51 & 54 & 47 & 54 & 50 & 51 & 54 & 52 & 51 & 51 & 47 & 45 &    &    &    &    &    &    &    &    &    & \\ \hline
58 & 58 & 46 & 58 & 50 & 58 & 52 & 58 & 51 & 58 & 54 & 58 & 51 & 58 & 55 & 58 & 53 & 58 & 50 & 58 & 48 & 58 & 53 & 58 & 52 & 58 & 54 & 58 & 56 &    &    &    &    &    &    &    &    &    & \\ \hline
59 & 58 & 47 & 57 & 51 & 56 & 52 & 55 & 54 & 54 & 52 & 53 & 56 & 54 & 51 & 51 & 53 & 55 & 55 & 49 & 57 & 51 & 54 & 55 & 51 & 53 & 53 & 51 & 49 & 47 &    &    &    &    &    &    &    &    & \\ \hline
60 & 60 & 48 & 60 & 51 & 60 & 55 & 60 & 54 & 60 & 57 & 60 & 52 & 60 & 56 & 60 & 56 & 60 & 55 & 60 & 49 & 60 & 54 & 60 & 53 & 60 & 55 & 60 & 57 & 60 &    &    &    &    &    &    &    &    & \\ \hline
61 & 60 & 48 & 59 & 53 & 58 & 54 & 57 & 55 & 56 & 54 & 55 & 58 & 54 & 54 & 53 & 53 & 56 & 56 & 53 & 59 & 50 & 57 & 55 & 54 & 57 & 57 & 56 & 54 & 50 & 48 &    &    &    &    &    &    &    & \\ \hline
62 & 62 & 49 & 62 & 53 & 62 & 56 & 62 & 57 & 62 & 58 & 62 & 54 & 62 & 57 & 62 & 59 & 62 & 55 & 62 & 51 & 62 & 54 & 62 & 55 & 62 & 56 & 62 & 58 & 62 & 60 &    &    &    &    &    &    &    & \\ \hline
63 & 62 & 50 & 61 & 54 & 60 & 56 & 59 & 56 & 58 & 56 & 57 & 59 & 56 & 57 & 55 & 54 & 56 & 57 & 56 & 60 & 52 & 60 & 57 & 57 & 60 & 56 & 56 & 56 & 56 & 52 & 50 &    &    &    &    &    &    & \\ \hline
64 & 64 & 51 & 64 & 56 & 64 & 57 & 64 & 60 & 64 & 59 & 64 & 57 & 64 & 58 & 64 & 62 & 64 & 58 & 64 & 56 & 64 & 53 & 64 & 60 & 64 & 57 & 64 & 59 & 64 & 61 & 64 &    &    &    &    &    &    & \\ \hline
65 & 64 & 52 & 63 & 56 & 62 & 59 & 61 & 58 & 60 & 59 & 59 & 60 & 58 & 60 & 58 & 56 & 56 & 58 & 60 & 61 & 54 & 63 & 56 & 60 & 60 & 57 & 60 & 61 & 60 & 57 & 54 & 52 &    &    &    &    &    & \\ \hline
66 & 66 & 52 & 66 & 57 & 66 & 58 & 66 & 61 & 66 & 60 & 66 & 60 & 66 & 59 & 66 & 63 & 66 & 61 & 66 & 60 & 66 & 54 & 66 & 59 & 66 & 58 & 66 & 60 & 66 & 62 & 66 & 64 &    &    &    &    &    & \\ \hline
67 & 66 & 53 & 65 & 58 & 64 & 61 & 63 & 60 & 62 & 62 & 61 & 61 & 60 & 63 & 62 & 58 & 58 & 60 & 63 & 62 & 58 & 65 & 55 & 63 & 60 & 60 & 64 & 61 & 59 & 59 & 60 & 55 & 53 &    &    &    &    & \\ \hline
68 & 68 & 54 & 68 & 58 & 68 & 60 & 68 & 62 & 68 & 61 & 68 & 63 & 68 & 60 & 68 & 64 & 68 & 64 & 68 & 61 & 68 & 56 & 68 & 59 & 68 & 62 & 68 & 61 & 68 & 63 & 68 & 65 & 68 &    &    &    &    & \\ \hline
69 & 68 & 55 & 67 & 60 & 66 & 62 & 65 & 63 & 64 & 65 & 63 & 62 & 62 & 66 & 63 & 62 & 60 & 60 & 63 & 63 & 61 & 66 & 57 & 66 & 63 & 63 & 65 & 60 & 63 & 64 & 63 & 60 & 57 & 55 &    &    &    & \\ \hline
70 & 70 & 56 & 70 & 60 & 70 & 63 & 70 & 63 & 70 & 62 & 70 & 66 & 70 & 61 & 70 & 65 & 70 & 67 & 70 & 63 & 70 & 61 & 70 & 58 & 70 & 65 & 70 & 62 & 70 & 64 & 70 & 66 & 70 & 68 &    &    &    & \\ \hline
71 & 70 & 56 & 69 & 61 & 68 & 63 & 67 & 66 & 66 & 67 & 65 & 63 & 65 & 68 & 63 & 64 & 62 & 61 & 63 & 64 & 65 & 67 & 59 & 69 & 62 & 66 & 65 & 63 & 67 & 66 & 64 & 64 & 63 & 59 & 56 &    &    & \\ \hline
72 & 72 & 57 & 72 & 63 & 72 & 66 & 72 & 64 & 72 & 64 & 72 & 69 & 72 & 63 & 72 & 66 & 72 & 70 & 72 & 66 & 72 & 65 & 72 & 59 & 72 & 64 & 72 & 64 & 72 & 66 & 72 & 67 & 72 & 69 & 72 &    &    & \\ \hline
73 & 72 & 58 & 71 & 63 & 70 & 65 & 69 & 68 & 68 & 68 & 67 & 65 & 66 & 69 & 65 & 67 & 65 & 63 & 63 & 67 & 69 & 68 & 63 & 71 & 60 & 69 & 65 & 66 & 70 & 65 & 66 & 67 & 66 & 65 & 60 & 58 &    & \\ \hline
74 & 74 & 59 & 74 & 64 & 74 & 67 & 74 & 66 & 74 & 67 & 74 & 70 & 74 & 66 & 74 & 67 & 74 & 71 & 74 & 69 & 74 & 68 & 74 & 61 & 74 & 64 & 74 & 69 & 74 & 66 & 74 & 68 & 74 & 70 & 74 & 72 &    & \\ \hline
75 & 74 & 60 & 73 & 65 & 72 & 67 & 71 & 69 & 70 & 69 & 69 & 67 & 68 & 70 & 67 & 70 & 69 & 65 & 65 & 67 & 70 & 69 & 68 & 72 & 62 & 72 & 68 & 69 & 70 & 66 & 70 & 71 & 69 & 69 & 66 & 62 & 60 & \\ \hline
76 & 76 & 60 & 76 & 65 & 76 & 68 & 76 & 69 & 76 & 70 & 76 & 71 & 76 & 69 & 76 & 68 & 76 & 72 & 76 & 72 & 76 & 68 & 76 & 66 & 76 & 63 & 76 & 70 & 76 & 67 & 76 & 69 & 76 & 71 & 76 & 73 & 76 & \\ \hline
77 & 76 & 61 & 75 & 67 & 74 & 70 & 73 & 70 & 72 & 70 & 71 & 70 & 70 & 71 & 69 & 73 & 72 & 69 & 67 & 67 & 70 & 70 & 70 & 73 & 64 & 75 & 67 & 72 & 70 & 69 & 74 & 70 & 69 & 70 & 69 & 69 & 64 & 61\\ \hline
\end{tabular}
}
\normalsize
\begin{center}
Table $1$: Independence number of $P(n,k), n \leq 77$.
\end{center}
\section*{Acknowledgement}
The authors would like to thank the referee for careful reading
of this paper and very helpful comments. And the authors like to
thank professor E. S. Mahmoodian for suggesting the problem and
very useful comments. We also thank  Nima Aghdaei, and Hadi
Moshaiedi for their computer program and algorithm to create
presented  table of $\alpha(P(n,k))$.
\def\cprime{$'$} \def\cprime{$'$} \def\cprime{$'$}

\end{document}